\documentclass[11pt,a4paper]{article}
\usepackage[english]{babel}
\usepackage{amsmath}
\usepackage{amsfonts}
\usepackage{amssymb}
\usepackage{xcolor}
\usepackage{amsthm}
\usepackage{array} 
\usepackage{longtable}
\usepackage{multirow}
\usepackage{fullpage}

\newtheorem{lemma}{Lemma}
\newtheorem{theorem}{Theorem}
\newtheorem{remark}{Remark}

\title{On linear independence measures of the values of Mahler functions}
\author{Keijo V\"a\"an\"anen \and Wen Wu}
\date{\today}

\begin{document}
\maketitle
\begin{abstract}
In this paper, we estimate the linear independence measures for the values of a class Mahler functions of degree one and two. For the purpose, we study the determinants of suitable Hermite-Pad\'{e} approximation polynomials. Based on the non-vanishing of these determinants, we apply the functional equations to get an infinite sequence of approximations which is used to produce the linear independence measures.
\end{abstract}

\section{Introduction and results}
In the present work our aim is to obtain linear independence measures for the values of a class of Mahler functions $F(z), G(z)\in\mathbb{Q}[[z]]$ converging on some open disc $D_{r}:=\{z:|z|<r\leq 1\}$ and satisfying a system of Mahler type functional equations
\begin{equation}\label{eq:Mahler type}
\left\{
\begin{aligned}
F(z^d) & = p_{11}(z)F(z) + p_{12}(z)G(z) + p_{10}(z),\\
G(z^d) & = p_{21}(z)F(z) + p_{22}(z)G(z) + p_{20}(z)
\end{aligned}
\right.
\end{equation}
with $p_{ij}(z) \in \mathbb{Q}(z)$ satisfying $p_{11}(z)p_{22}(z) - p_{12}(z)p_{21}(z) \neq 0.$ Note that Mahler functions of degree one or two satisfy functional equations of the above type, if $F(z)$ and $G(z)$ are Mahler functions of degree one, then $p_{12}(z) = p_{21}(z) = 0$, and if $F(z)$ is of degree two, then we choose $G(z) = F(z^d)$. Our general result (Theorem 6 in Sec. 4) needs some technical notations to be presented later, and therefore to introduce our results we give here applications to some well-known  functions.

The linear independence measures studied here are lower bounds for linear forms (in $1$ and certain numbers $\gamma_{1}$ and $\gamma_{2}$) of the form 
\begin{equation}
|h_{0}+h_{1}\gamma_{1}+h_{2}\gamma_{2}|>CH^{-\mu} \label{e:lower}
\end{equation}
valid for any integers $h_{0}$, $h_{1}$, $h_{2}$, not all zero, where the exponent $\mu$ is given explicitly, $ H=\max\{|h_{1}|,|h_{2}|, H_{0}\}$, and positive constants $C$ and $H_{0}$ are independent of $h_{i}$. In our results $\gamma_{1}$ and $\gamma_{2}$ are the values of the functions under consideration at rational points $a/b\in D_{r}\backslash \{0\}$, where $\log |a| = \lambda \log b~ (0\leq \lambda <\log (rb)/\log b)$. We note that generally \cite[Theorem  4.4.1]{Nishioka} implies the existence of a $\mu ~ (\geq 2)$ in our cases below, and here our aim is to obtain an explicit upper bound for the linear independence exponent
\[\mu(\gamma_{1},\gamma_{2}):=\inf\{\mu:(\ref{e:lower}) \text{ holds for some } C>0, H_{0}>0\}.\] 

This work is a continuation to \cite{VM2015}, where we studied simultaneous approximations of similar numbers $\gamma_{1}$ and $\gamma_{2}$. We also note that, after Bugeaud's remarkable work \cite{Bugeaud2011} on Thue-Morse numbers, there has appeared several works on the irrationality exponents of the values of degree one Mahler functions, see \cite{AR2009, BHWY2015, Coons2013, GWW2014, Keijo2015, WW2014} and the references in \cite{BHWY2015}. In particular, the irrationality exponents of the numbers in Theorem 1-3 below equal $2$.

\subsection{Thue-Morse number and its square}
Our first result studies the product 
\begin{align*}
T(z)=\prod_{j=0}^{\infty}(1-z^{2^{j}}),
\end{align*}
the generating function of the Thue-Morse sequence on $\{-1,1\}$, satisfying 
\begin{align}
T(z)=(1-z)T(z^{2}). \label{Thue}
\end{align}
\begin{theorem}\label{thm:1}
We have \[\mu\left(T\left(\frac{1}{b}\right),T^{2}\left(\frac{1}{b}\right)\right)\leq \frac{91}{32}\approx 2.843\dots.\] More generally, if $0\leq \lambda <7/29$, then 
\begin{align*}
\mu\left(T\left(\frac{a}{b}\right),T^{2}\left(\frac{a}{b}\right)\right)\leq \frac{91}{32-104\lambda}.
\end{align*}
\end{theorem}

\subsection{Stern's sequence and its twisted version}
Next, let $A(z)$ and $B(z)$ be generating functions of Stern's diatomic sequence and its twisted version. These functions satisfy functional equations 
\begin{align}
A(z)=(1+z+z^{2})A(z^{2}), \qquad
B(z)=2-(1+z+z^{2})B(z^{2}), \label{Stern}
\end{align}
of type (I), see e.g. \cite{BV2013}.
\begin{theorem}\label{thm:2}
We have \[\mu\left(A\left(\frac{1}{b}\right), B\left(\frac{1}{b}\right)\right)\leq \frac{26}{9}\approx 2.888\dots.\]
More generally,  
\begin{equation*}
\mu\left(A\left(\frac{a}{b}\right), B\left(\frac{a}{b}\right)\right)\leq \left\{
\begin{aligned}
\frac{130}{45-149\lambda}, &\quad \textrm{ if }\lambda<\frac{145}{1289},\\
\frac{69}{25-89\lambda}, &\quad \textrm{ if }\frac{145}{1289}\leq \lambda <\frac{5}{29}.\\
\end{aligned}\right. 
\end{equation*}
\end{theorem}

\subsection{Lambert series $G_{3}(z)$ and $F_{3}(z)$}
The functions 
\begin{align*}
G_{3}(z)=\sum_{j=0}^{\infty}\frac{z^{3^{j}}}{1-z^{3^{j}}}, \qquad F_{3}(z)=\sum_{j=0}^{\infty}\frac{z^{3^{j}}}{1+z^{3^{j}}}=-G_{3}(-z)
\end{align*}
satisfy
\begin{align}
(1-z)G_{3}(z)-(1-z)G_{3}(z^{3})-z=0,\qquad
(1+z)F_{3}(z)-(1+z)F_{3}(z^{3})-z=0. \label{G3F3}
\end{align}
The following result studies the values of these typical examples of Mahler functions.
\begin{theorem}\label{thm:3}
We have \[\mu\left(G_{3}\left(\frac{1}{b}\right), F_{3}\left(\frac{1}{b}\right)\right)\leq \frac{129}{37}\approx 3.486\dots.\]
More generally,  
\begin{equation*}
\mu\left(G_{3}\left(\frac{a}{b}\right), F_{3}\left(\frac{a}{b}\right)\right)\leq\left\{
\begin{aligned}
\frac{129}{37-119\lambda},& \quad \textrm{ if } \lambda < \frac{25}{443},\\
\frac{83}{24-80\lambda}, &\quad \textrm{ if } \frac{25}{443}\leq \lambda < \frac{43}{337},\\ 
\frac{57}{17-59\lambda}, &\quad \textrm{ if } \frac{43}{337}\leq \lambda < \frac{7}{29}.
\end{aligned}
\right.
\end{equation*}
\end{theorem}

\subsection{The Rudin-Shapiro sequence}
Let $(r_{n})_{n\geq 0}$ be the Rudin-Shapiro sequence defined by \(r_0 = 1, r_{2n} = r_n, r_{2n+1} = (-1)^nr_n.\)
Its generating function $R(z)=\sum_{n\geq 0}r_{n}z^{n}$ satisfies
\begin{equation}\label{e:rudin}
R(z) = R(z^2) + zR(-z^2).
\end{equation}
We shall investigate the values of $R(z)$ and $R(-z)$ at some rational points.
\begin{theorem}\label{thm:rudin}
We have \[\mu\left(R\left(\frac{1}{b}\right), R\left(-\frac{1}{b}\right)\right)\leq \frac{13}{4}=3.25.\]
More generally,  
\begin{equation*}
\mu\left(R\left(\frac{a}{b}\right), R\left(-\frac{a}{b}\right)\right)\leq\left\{
\begin{aligned}
\frac{39}{12-40\lambda},& \quad \textrm{ if }~ \lambda < \frac{21}{187},\\
\frac{47}{15-53\lambda}, &\quad \textrm{ if }~ \frac{21}{187}\leq \lambda < \frac{3}{13}.
\end{aligned}
\right.
\end{equation*}
\end{theorem}

\subsection{A degree $2$ Mahler function}
As an example of degree $2$ Mahler functions we take the function $S(z)$ satisfying $S(0)=1$ and 
\begin{align}
zS(z)-(1+z+z^{2})S(z^{4})+S(z^{16})=0. \label{degree2}
\end{align}
This function was introduced by Dilcher and Stolarsky \cite{DS2009}, and it has been studied recently in several works , see e.g. \cite{Adam2010}, \cite{BCZ} and \cite{BV2014}, in particular the algebraic independence of $S(\alpha)$, $S^{\prime}(\alpha)$, $S(\alpha^{4})$ and  $S^{\prime}(\alpha^{4})$ is proved in \cite{BCZ} for all algebraic $\alpha$, $0<|\alpha|<1$. Note also that in \cite{BV2014} an upper bounded $5$ is obtained for the irrationality exponent of $S(1/b)$.
\begin{theorem}\label{thm:4}
We have \[\mu\left(S\left(\frac{1}{b}\right), S\left(\frac{1}{b^{4}}\right)\right)\leq \frac{167}{25}= 6.68.\]
More generally, if $0\leq \lambda < 1/5$, then 
\begin{align*}
\mu\left(S\left(\frac{a}{b}\right), S\left(\left(\frac{a}{b}\right)^{4}\right)\right)\leq \frac{167}{25-93\lambda}.
\end{align*}
\end{theorem}

All results above are based on non-vanishing of the determinants of suitable Hermite-Pad\'{e} approximation polynomials studied in Section $2$. This non-vanishing is verified here by computing the determinants, but it would be of great interest to find a more general criterion for this. After having some non-zero determinants the functional equations can be used to produce a sufficiently dense infinite sequence of approximations with non-zero determinants. It is well-known that such approximations can be used to produce linear independence measures. Section $3$ contains this consideration, and it is then applied to prove a general result in Section $4$. The proofs of Theorems $1$-$5$ are given in Section $5$.

\section{Important determinants}
We first note that the above system \eqref{eq:Mahler type} can be given in the form
\begin{align}
P(z)F(z^d) & = P_{11}(z)F(z) + P_{12}(z)G(z) + P_{10}(z),\label{1}\\
P(z)G(z^d) & = P_{21}(z)F(z) + P_{22}(z)G(z) + P_{20}(z),\label{2}
\end{align}
where $P(z)$, the least common denominator of $p_{ij}(z)$, and $P_{ij}(z) = P(z)p_{ij}(z)$ belong to $\mathbb{Z}[z]$ and satisfy $P_{11}(z)P_{22}(z) - P_{12}(z)P_{21}(z) \neq 0.$

For an integer $k \geq 1$, let $A_k(z), B_k(z), C_k(z) \in \mathbb{Z}[z]$ denote $(d_1,d_2,d_3) = (d_1(k),d_2(k),d_3(k))$ Hermite-Pad\'{e} approximation polynomials of $F(z), G(z)$ and 1, so
\begin{equation}\label{3}
A_k(z)F(z) + B_k(z)G(z) + C_k(z) = R_k(z),
\end{equation}
where deg $A_k(z) \leq d_1$, deg $B_k(z) \leq d_2$, deg $C_k(z) \leq d_3$ and the order of zero at $z = 0$ of the remainder term $R_k(z)$ satisfies ord $R_k(z) =: o(k) \geq d_1 + d_2 + d_3 + 2$. Clearly such polynomials, where at least one of $A_k(z), B_k(z)$ is not zero, exist.
Substituting in (\ref{3}) $z^d$ for $z$ and applying (\ref{1}) and (\ref{2}), we obtain
\begin{eqnarray*}
 (P_{11}(z)A_{k}(z^{d}) + P_{21}(z)B_k(z^d))F(z) + (P_{12}(z)A_k(z^d) + P_{22}(z)B_{k}(z^{d}))G(z) &&\\
+ P_{10}(z)A_{k}(z^{d}) + P_{20}(z)B_{k}(z^{d})
+ P(z)C_{k}(z^{d}) = P(z)R_{k}(z^{d}).
\end{eqnarray*}
Repeating this procedure $m$ times, we have
\begin{equation}
A_{k,m}(z)F(z)+B_{k,m}(z)G(z)+C_{k,m}(z)=R_{k,m}(z),\ m=0,1,\dots, \label{4}
\end{equation}
where $A_{k,0}(z)=A_{k}(z)$, $B_{k,0}(z)=B_{k}(z)$, $C_{k,0}(z)=C_{k}(z)$, $R_{k,0}(z)=R_{k}(z)$ and, for $m=1,2,\dots$,
\begin{equation}\label{e:resursion}
\left\{
\begin{aligned}
A_{k,m}(z)  = & P_{11}(z)A_{k,m-1}(z^{d}) + P_{21}(z)B_{k,m-1}(z^d),\\
B_{k,m}(z)  = & P_{12}(z)A_{k,m-1}(z^{d}) + P_{22}(z)B_{k,m-1}(z^d),\\
C_{k,m}(z)  =&  P_{10}(z)A_{k,m-1}(z^{d}) + P_{20}(z)B_{k,m-1}(z^d)\\
 & + P(z)C_{k,m-1}(z^{d}),\\
R_{k,m}(z)  = & P(z)R_{k,m-1}(z^{d}).
\end{aligned}
\right.
\end{equation}

We are interested in determinants
\[
\Delta(\underline{k},m,z) := \det\left(\begin{array}{ccc} 
  A_{k_1,m}(z) & B_{k_1,m}(z)& C_{k_1,m}(z) \\
	A_{k_2,m}(z) & B_{k_2,m}(z)& C_{k_2,m}(z) \\
	A_{k_3,m}(z) & B_{k_3,m}(z)& C_{k_3,m}(z) \\
  \end{array}\right),
\] 
where $1 \leq k_1 < k_2 < k_3.$ By the above recursions (\ref{e:resursion})
\[
\Delta(\underline{k},m,z) = \Phi(z)\Delta(\underline{k},m-1,z^d),\quad \Phi(z) := (P_{11}(z)P_{22}(z) - P_{12}(z)P_{21}(z))P(z),
\]
and so
\begin{equation}\label{5}
\Delta(\underline{k},m,z) = \Delta(\underline{k},0,z^{d^m})\prod_{j=0}^{m-1}\Phi(z^{d^j}).
\end{equation}
In particular, for degree one functions we have $\Phi(z) = P_{11}(z)P_{22}(z)P(z)$, and for degree two function $F(z)$ with $G(z) = F(z^d)$ the function $\Phi(z) = P_{21}(z)P^2(z)$.

Let $\bar{d}(k) := \max \{d_1(k),d_2(k),d_3(k)\}$. By our assumption $k_1 < k_2 < k_3$ it is natural to assume that $\bar{d}(k_1) \leq \bar{d}(k_2) \leq \bar{d}(k_3)$ and $o(k_1) \leq o(k_2) \leq o(k_3)$. Since
\[
\Delta(\underline{k},0,z) = \det\left(\begin{array}{ccc} 
  A_{k_1}(z) & B_{k_1}(z)& R_{k_1}(z) \\
	A_{k_2}(z) & B_{k_2}(z)& R_{k_2}(z) \\
	A_{k_3}(z) & B_{k_3}(z)& R_{k_3}(z) \\
  \end{array}\right),
\]
it follows that $o(k_1) \leq$ ord $\Delta(\underline{k},0,z) \leq \deg \Delta(\underline{k},0,z) \leq \bar{d}(k_1) + \bar{d}(k_2) + \bar{d}(k_3)$, if $\Delta(\underline{k},0,z) \neq 0$. Thus in this case
\begin{equation}\label{6}
\Delta(\underline{k},0,z) =: z^{o(k_1)}D(\underline{k},z)
\end{equation} 
with some polynomial $D(\underline{k},z) \neq 0, \deg D(\underline{k},z) \leq \bar{d}(k_1) + \bar{d}(k_2) + \bar{d}(k_3) - o(k_1)$. Further,
if $o(k_1) > \bar{d}(k_1) + \bar{d}(k_2) + \bar{d}(k_3)$, then $\Delta(\underline{k},0,z) = 0$.

We note that the condition $D(\underline{k},z) \neq 0$ gives a strong restriction to $o(k_1)$. For example, if $d_j(k_1) = k, d_j(k_2) = k+1, d_j(k_3) = k+2\ (j = 1,2,3)$, then deg $\Delta(\underline{k},0,z) \leq 3k+3$ and $o(k_1) \geq 3k+2$. Thus the condition $D(\underline{k},z) \neq 0$ is possible only if $3k+2 \leq o(k_1) \leq 3k+3$.

The above means that one determinant $\Delta(\underline{k},0,z) \neq 0$ gives an infinite sequence of determinants $\Delta(\underline{k},m,z) \neq 0, m = 0,1,\ldots .$ When considering the values of the functions at rational points $z = a/b$ we need to know that $\Delta(\underline{k},m,a/b) \neq 0$ at least for all sufficiently large $m$. This condition can be verified in many concrete cases by using (\ref{5}) and (\ref{6}), since deg $D(\underline{k},z)$ is small.

\section{Fundamental lemma}

In this section $\gamma_1$ and $\gamma_2$ denote real numbers and $b \geq 2$ is an integer. Let $\underline{k} = \underline{k}(\ell) = (k_{\ell,1},k_{\ell,2},k_{\ell,3})$ $(\ell = 1,\ldots,L)$ be vectors with positive integer components $k_{\ell,i}$ satisfying $k_{\ell,1}<k_{\ell,2}<k_{\ell,3}$ and $k_{\ell,3}\leq k_{\ell+1,1}\ (\ell = 1,\ldots,L-1), k_{L,3}\leq dk_{1,1}$. Assume that for each $k = k_{\ell,i}$ there exists an integer $m_0(k)$ such that for all $m \geq m_0(k)$ we have linear forms
\[
a_{k,m}\gamma_1 + b_{k,m}\gamma_2 + c_{k,m} = r_{k,m}
\]
with the following properties (i) -- (iii).
\begin{itemize}
\item[(i)] The coefficients $a_{k,m}, b_{k,m},c_{k,m} \in \mathbb{Z}$ and satisfy
\begin{equation}\label{8}
\max \{\left|a_{k,m}\right|,\left|b_{k,m}\right|\} \leq c_1(k)b^{E(k)d^m},
\end{equation}
where $E(k)$ and $c_1(k)$ (as also $c_2(k),\ldots $ later) are positive constants independent of $m$.
\item[(ii)] We have
\begin{equation}\label{9}
\left|r_{k,m}\right| \leq c_2(k)b^{-V(k)d^m},
\end{equation}
where $V(k) > 0$ is independent of $m$.
\item[(iii)] The determinant
\[
\det\left(\begin{array}{ccc} 
  a_{k_{\ell,1},m} & b_{k_{\ell,1},m}& c_{k_{\ell,1},m} \\
	a_{k_{\ell,2},m} & b_{k_{\ell,2},m}& c_{k_{\ell,2},m} \\
	a_{k_{\ell,3},m} & b_{k_{\ell,3},m}& c_{k_{\ell,3},m} \\
\end{array}\right)
\neq 0.
\]
for all $\ell = 1,\ldots,L; m \geq m_0(\underline{k}(\ell)) = \max_{1\leq i\leq 3}\{m_0(k_{\ell,i})\}$.
\end{itemize}

For the following fundamental lemma, we finally define, for all $\ell = 1,\ldots,L$, the notations
\begin{align*}
\theta(\ell) & = \max_{1\leq i<j \leq 3} \{E(k_{\ell,i}) + E(k_{\ell,j})\},\\
\nu(\ell) & = \min_{\substack{ 1\leq i,j \leq 3 \\ i\neq j}} \{V(k_{\ell,i}) - E(k_{\ell,j})\},
\end{align*}
and denote $K := (\underline{k}(1),\ldots,\underline{k}(L))$.

\begin{lemma}\label{lem:1}
Suppose that $0 < \nu(1) <\cdots< \nu(L) < d\nu(1)$. Then there exist positive constants $C = C(K)$ and $H_0 = H_0(K)$ such that for any integers $h_0, h_1, h_2$, not all zero,
\[
\left|h_0 + h_1\gamma_1 + h_2\gamma_2\right| > CH^{-\mu},
\]
where $H = \max \{\left|h_1\right|, \left|h_2\right|, H_0\}$ and
\[
\mu = \max_{1\leq \ell \leq L} \mu(\ell), \quad \mu(\ell) := \frac{\theta(\ell +1)}{\nu(\ell)}, \quad \theta(L+1) := d\theta(1).\\  
\]
\end{lemma}

\begin{proof}
Let
\[
\Lambda = h_0 + h_1\gamma_1 + h_2\gamma_2.
\]
By the condition (iii) above, for all $\ell = 1,\ldots,L$ there exist $1\leq i < j\leq 3$ such that
\[
D(\underline{k}(\ell),\underline{h}) := \det\left(\begin{array}{ccc} 
  h_1 & h_2& h_0 \\
	a_{k_{\ell,i},m} & b_{k_{\ell,i},m}& c_{k_{\ell,i},m} \\
	a_{k_{\ell,j},m} & b_{k_{\ell,j},m}& c_{k_{\ell,j},m} \\
  \end{array}\right)
=\det\left(\begin{array}{ccc} 
  h_1 & h_2& \Lambda \\
	a_{k_{\ell,i},m} & b_{k_{\ell,i},m}& r_{k_{\ell,i},m} \\
	a_{k_{\ell,j},m} & b_{k_{\ell,j},m}& r_{k_{\ell,j},m} \\
  \end{array}\right)
\neq 0.		
\]
Since $D(\underline{k}(\ell),\underline{h})$ is an integer, we obtain, by (\ref{8}) and (\ref{9}),
\begin{equation}\label{10}
1 \leq 2\left|\Lambda\right|c_1(k_{\ell,i})c_1(k_{\ell,j})b^{(E(k_{\ell,i})+E(k_{\ell,j}))d^m}+2hc_1(k_{\ell,j})c_2(k_{\ell,i})b^{-(V(k_{\ell,i})-E(k_{\ell,j}))d^m}+
\end{equation}
\[
2hc_1(k_{\ell,i})c_2(k_{\ell,j})b^{-(V(k_{\ell,j})-E(k_{\ell,i}))d^m}
\]
with $h = \max \{\left|h_1\right|,\left|h_2\right|\}$. The definitions of $\theta(\ell)$ and $\nu(\ell)$ then give
\begin{equation}\label{11}
1 \leq C_1(K)\left|\Lambda\right|b^{\theta(\ell)d^m} + C_2(K)hb^{-\nu(\ell)d^m}
\end{equation}
for all $m \geq M_0 := \max \{m_0(\underline{k}(1)),\ldots,m_0(\underline{k}(L))\}$, and here $C_1(K)$ and $C_2(K)$ (and also $C_3(K)$ later) are positive constants depending on $K$. Note that $C_1(K)$ and $C_2(K)$ are the same for all $\ell$.

We now choose $H_0$ in such a way that
\[
2C_2(K)H_0 \geq b^{\nu(1)d^{M_0}},
\]
and fix the pair $(\ell,m)$ from the sequence $(1,M_0)$, $\ldots$, $(L,M_0)$,  $(1,M_0+1)$, $\ldots$, $(L,M_0+1)$, $(1,M_0+2)$, $\ldots$ to be the first one satisfying
\[
2C_2(K)H < b^{\nu(\ell)d^m},
\]
where $H = \max \{h,H_0\}$. Then $(\ell,m) \neq (1,M_0)$, and the pair just before it is $(\ell-1,m)$, if $\ell > 1$, and $(L,m-1)$, if $\ell = 1$. The above choice means that
\[
2C_2(K)H \geq b^{\nu(\ell-1)d^m}, \quad \ell > 1,
\]
\[
2C_2(K)H \geq b^{\nu(L)d^{m-1}}, \quad \ell = 1.
\]

In the first case, by (\ref{11}),
\[
\frac{1}{2} < C_1(K)\left|\Lambda\right|b^{\theta(\ell)d^m} = C_1(K)\left|\Lambda\right|(b^{\nu(\ell-1)d^m})^{\theta(\ell)/\nu(\ell-1)} \leq C_3(K)\left|\Lambda\right|H^{\mu}.
\]
In the case $\ell = 1$ we similarly have
\[
\frac{1}{2} < C_1(K)\left|\Lambda\right|b^{\theta(1)d^m} = C_1(K)\left|\Lambda\right|(b^{\nu(L)d^{m-1}})^{d\theta(1)/\nu(L)} \leq C_3(K)\left|\Lambda\right|H^{\mu}.
\]
This proves our lemma.
\end{proof}

\section{General theorem}

We now assume that $F(z), G(z) \in \mathbb{Q}[[z]]$ converge in some disk $D_r$ and satisfy (\ref{1}) and (\ref{2}). Our aim is to apply Lemma \ref{lem:1} to consider the function values $F(a/b)$ and $G(a/b)$ at non-zero rational points $a/b \in D_r$, where $\log \left|a\right| = \lambda \log b$, $0 \leq \lambda < \log(rb)/\log b$. We also assume that
\begin{equation}\label{12}
(P_{11}((a/b)^{d^j})P_{22}((a/b)^{d^j}) - P_{12}((a/b)^{d^j})P_{21}((a/b)^{d^j}))P((a/b)^{d^j}) \neq 0, \ j = 0,1,\ldots .
\end{equation}

The approximation forms we use are obtained from (\ref{4}) at $z = a/b$. The recursions (\ref{e:resursion}) imply, for all $m \geq 1$,
\begin{equation}\label{13}
\deg A_{k,m}(z),\, \deg B_{k,m}(z),\, \deg C_{k,m}(z) \leq \left(\bar{e}(k)+\frac{\tau}{d-1}\right)\cdot d^{m}-\frac{\tau}{d-1}
\end{equation}
where $\bar{e}(k)$ and $\tau$ are non-negative integers satisfying $\bar{e}(k) \leq \bar{d}(k) := \max \{d_1(k),d_2(k),d_3(k)\}$ and $\tau \leq \nu$, the maximum of the degrees of $P_{ij}(z)$ and $P(z)$. Thus the multiplication of (\ref{4}) at $z = a/b$ by 
\[
Q_{k,m} := b^{(\bar{e}(k)+\frac{\tau}{d-1})d^m-\frac{\tau}{d-1}}
\]
leads to linear forms
\[
a_{k,m}F(\frac{a}{b}) + b_{k,m}G(\frac{a}{b}) + c_{k,m} = r_{k,m}, \quad m = 0, 1,\ldots,
\]
where all $a_{k,m}, b_{k,m}$ and $c_{k,m}$ are integers. To be able to apply Lemma 1 with $\gamma_1 = F(a/b), \gamma_2 = G(a/b)$ we need to estimate the coefficients $a_{k,m}$ and $b_{k,m}$ and the remainders $r_{k,m}$. For this we apply the recursions (\ref{e:resursion}).

Let $\widetilde{P}(z)$ denote the polynomial, where the coefficient of $z^j$ is the maximum of the absolute values of the corresponding coefficients in $P_{ij}(z), 1 \leq i,j \leq 2$. Then, for all $m = 1,2,\ldots$, 
\[
\left|A_{k,m}(z)\right| \leq \widetilde{P}(\left|z\right|)(\left|A_{k,m-1}(z^d)\right| + \delta\left|B_{k,m-1}(z^d)\right|),
\]
\[
\left|B_{k,m}(z)\right| \leq \widetilde{P}(\left|z\right|)(\delta\left|A_{k,m-1}(z^d)\right| + \left|B_{k,m-1}(z^d)\right|),
\]
where $\delta = 0$ for degree one functions $F(z)$ and $G(z)$, and $\delta = 1$ otherwise. Applying these inequalities we obtain
\[
\max \{\left|A_{k,m}(z)\right|, \left|B_{k,m}(z)\right|\} \leq (1+\delta)^m\max \{\left|A_k(z^{d^m})\right|, \left|B_k(z^{d^m})\right|\}\prod_{j=0}^{m-1}\widetilde{P}(\left|z\right|^{d^j}).
\]
Therefore, for all $m \geq m_1(k)$,
\[
\max \{\left|a_{k,m}\right|, \left|b_{k,m}\right|\} \leq c_3(k)b^{(\bar{e}(k)+\frac{\tau}{d-1})d^m},
\]
if the condition
\begin{equation}\label{e1}
(1+\delta)\left|\widetilde{P}(0)\right| \leq 1
\end{equation}
holds. Generally, for any given $\delta_1 > 0$,
\begin{equation}\label{e2}
\max \{\left|a_{k,m}\right|, \left|b_{k,m}\right|\} \leq c_3(k)b^{(\bar{e}(k)+\frac{\tau}{d-1}+\delta_1)d^m}
\end{equation}
for all $m \geq m_2(k,\delta_1)$, and under the condition (\ref{e1}) we may choose here $\delta_1 = 0$.

Since
\[
R_{k,m}(z) = R_k(z^{d^m})\prod_{j=0}^{m-1}P(z^{d^j}),
\]
we also have
\[
\left|r_{k,m}\right| \leq c_4(k)\max\{1,\left|P(0)\right|^m\}b^{-((1-\lambda)o(k)-\bar{e}(k)-\frac{\tau}{d-1})d^m}
\]
for all $m \geq m_3(k)$. Thus, for any given $\delta_2 > 0$,
\begin{equation}\label{e3}
\left|r_{k,m}\right| \leq c_4(k)b^{-((1-\lambda)o(k)-\bar{e}(k)-\frac{\tau}{d-1}-\delta_2)d^m}
\end{equation}
for all $m \geq m_4(k,\delta_2)$, and we may use here the value $\delta_2 = 0$, if the condition
\begin{equation}\label{e4}
\left|P(0)\right| \leq 1
\end{equation}
holds.

Thus we have the estimates (\ref{8}) and (\ref{9}) for all $m \geq m_5(k,\delta_1,\delta_2)$, where
\begin{equation}\label{e5}
E(k) = \bar{e}(k)+\frac{\tau}{d-1}+\delta_1,~ V(k) = (1-\lambda)o(k)-\bar{e}(k)-\frac{\tau}{d-1}-\delta_2.
\end{equation}
By using these values with Lemma \ref{lem:1} we get the following theorem, we only need to note that the condition $D(\underline{k},z) \neq 0$ implies $D(\underline{k},(a/b)^{d^m}) \neq 0$ for all $m \geq m_6(\underline{k},a/b)$.

\begin{theorem}\label{thm:5}
Assume that the condition (\ref{12}) holds and $D(\underline{k},z) \neq 0$ for all $\ell = 1,\ldots,L.$ Let $\theta(\ell)$ and $\nu(\ell)$ be defined as in Lemma 1 with $E(k)$ and $V(k)$ given in (\ref{e5}). If $0 < \nu(1) <\cdots< \nu(L) < d\nu(1)$, then there exist positive constants $\lambda_0 = \lambda_0(K,F,G), C = C(K,a/b,F,G)$ and $H_0 = H_0(K,a/b,F,G)$ such that for all $0 \leq \lambda < \lambda_0$ and any integers $h_0, h_1, h_2$, not all zero,
\[
\left|h_0 + h_1F(\frac{a}{b}) + h_2G(\frac{a}{b})\right| > CH^{-\mu}
\]
with $H$ and $\mu$ as in Lemma \ref{lem:1}.
\end{theorem}

\section{Proof of Theorems \ref{thm:1}-\ref{thm:4}}
We are ready to prove Theorems \ref{thm:1}-\ref{thm:4}. We start by giving the following formulae which follow from (\ref{e5}):
\begin{equation}
\left\{
\begin{aligned}
\theta(\ell)&=\max_{1\leq i < j\leq 3}\left\{\bar{e}(k_{\ell,i})+\bar{e}(k_{\ell,j})\right\}+\frac{2\tau}{d-1}+2\delta_{1},\\
\nu(\ell)&=\min_{\substack{1\leq i, j\leq 3\\
i\neq j}}\left\{(1-\lambda)o(k_{\ell,i})-\bar{e}(k_{\ell,i})-\bar{e}(k_{\ell,j})\right\}-\frac{2\tau}{d-1}-\delta_{1}-\delta_{2}.
\end{aligned}
\right. \label{16}
\end{equation}
Thus we should choose $\tau$ and $\delta_{i}$ as small as possible while applying Lemma \ref{lem:1}. 

\begin{proof}[Proof of Theorem \ref{thm:1}]
To prove Theorem \ref{thm:1}, we apply Theorem \ref{thm:5} with $F(z)=T(z)$, $G(z)=T^{2}(z)$. Now, by (\ref{Thue}), 
\begin{align*}
(1-z)^{2}F(z^{2})=(1-z)F(z),\quad (1-z)^{2}G(z^{2})= G(z).
\end{align*}
Therefore $r=1$, $\delta=0$, $P(z)=(1-z)^{2}$, $\widetilde{P}(z)=1+z$ and $\widetilde{P}(0)=P(0)=1$ gives $\delta_{1}=\delta_{2}=0$. We shall use $(k, k+1, k-1)$ approximations and we may take $\bar{e}(k)=k+1$, $\tau=0$. Our $\underline{k}(\ell)$ are $(k_{\ell,1}, k_{\ell,1}+1, k_{\ell,1}+2)$ and the choices for $k=k_{\ell,1}$ are $29, 31,34, 43$ and $49$. For all these values $o(k)=3k+2$. Since $\textrm{deg}\Delta(\underline{k}(\ell),z)\leq 3k+3$, we have $D(\underline{k}(\ell),z )=s_{\ell,0}+s_{\ell,1}z$, where
\begin{equation*}
s_{\ell, 0}=\textrm{det}\begin{pmatrix}
A_{k+1}(0) & B_{k+1}(0)\\
A_{k+2}(0) & B_{k+2}(0)
\end{pmatrix} c \neq 0
\end{equation*}
where $c$ is the coefficient of $z^{3k+2}$ in $R_{k}(z)$ (see Appendix).
In fact $s_{\ell,0}$  is nonzero in all of our cases, also in the proofs of Theorems \ref{thm:2}-\ref{thm:4}.
By using (\ref{16}) we get 
\begin{align*}
\theta(\ell)=2k+5, \quad \nu(\ell)=k-2-\lambda(3k+2)
\end{align*}
for all $\lambda < 2/3$. So we have the following table.
\begin{center}
\begin{tabular}{|>{$}c<{$}|>{$}c<{$}|>{$}c<{$}|>{$}c<{$}|>{$}c<{$}|>{$}c<{$}|}
\hline
\ell & 1 & 2 & 3 & 4 & 5 \\
\hline
k & 29 & 31 & 34 & 43 & 49 \\
\hline
\theta(\ell) & 63 & 67 & 73 & 91 & 103 \\
\hline
\nu(\ell) & 27-89\lambda & 29-95\lambda & 32-104\lambda & 41-131\lambda & 47-149\lambda\\
\hline
\end{tabular}
\end{center}
For the condition $0<\nu(1)<\cdots < \nu(5)<2\nu(1)$ we need to assume $\lambda<\lambda_{0}:=7/29\approx 0.241\dots$. 
When $\lambda<\lambda_{0}$, the comparison of $\mu(\ell)$ gives
\[\mu=\max_{\ell}\frac{\theta(\ell+1)}{\nu(\ell)}=\frac{\theta(4)}{\nu(3)}=\frac{91}{32-104\lambda}.\]
This prove Theorem \ref{thm:1}.
\end{proof}

To prove Theorem \ref{thm:2}-\ref{thm:4}, we need to modify the choices of parameters.
\begin{proof}[Proof of Theorem \ref{thm:2}]
Here we apply Theorem \ref{thm:5} with $F(z)=A(z)$, $G(z)=B(z)$, and the use of (\ref{Stern}) gives $r=1$, $\delta=0$, $P(z)=1+z+z^{2}$, $\widetilde{P}(z)=1$ and $\delta_{1}=\delta_{2}=0$. The $(k, k+1, k-1)$ approximations give $\bar{e}(k)=k+1$, $\tau=1$. By choosing $\underline{k}(l)$ as above, where $k=k_{\ell,1}$ are $29$, $31$, $34$, $38$, $43$ and $49$, we get $o(k)=3k+2$ and the determinants $D(\underline{k}(\ell),z )\neq 0$ (see Appendix). Further, \[\theta(\ell)=2k+7, \quad \nu(\ell)=k-4-\lambda(3k+2)\] for all $\lambda<2/3$, and this leads to the following table.
\begin{center}
\begin{tabular}{|>{$}c<{$}|>{$}c<{$}|>{$}c<{$}|>{$}c<{$}|>{$}c<{$}|>{$}c<{$}|>{$}c<{$}|}
\hline
\ell & 1 & 2 & 3 & 4 & 5 & 6\\
\hline
k & 29 & 31 & 34 & 38 & 43 & 49 \\
\hline
\theta(\ell) & 65 & 69 & 75 & 83 & 93 & 105 \\
\hline
\nu(\ell) & 25-89\lambda & 27-95\lambda & 30-104\lambda & 34-116\lambda & 39-131\lambda & 45-149\lambda\\
\hline
\end{tabular}
\end{center}
To satisfy the condition $0<\nu(1)<\cdots < \nu(6)<2\nu(1)$, we need to assume $\lambda<\lambda_{0}:=5/29\approx 0.172\dots$. After the comparison of $\mu(\ell)=\theta(\ell +1)/\nu(\ell)$ we see that 
\begin{equation*}
\mu=\max_{1\leq \ell \leq 6}\mu(\ell)=\left\{
\begin{aligned}
\mu(6)&=\frac{130}{45-149\lambda}, && \textrm{ if }\lambda<\frac{145}{1289},\\
\mu(1)&=\frac{69}{25-89\lambda}, && \textrm{ if }\frac{145}{1289}\leq \lambda <\frac{5}{29}.\\
\end{aligned}\right. 
\end{equation*}
This proves Theorem \ref{thm:2}.
\end{proof}
\begin{remark}
We note that here all determinants $D(\underline{k}(\ell),z)\neq 0$, $1\leq k\leq 50$. In all other theorems most of these determinants equal zero.
\end{remark}

\begin{proof}[Proof of Theorem \ref{thm:3}]
In this case we apply Theorem \ref{thm:5} with $d=3$, $F(z)=G_{3}(z)$ and $G(z)=F_{3}(z)$. Then (\ref{G3F3}) implies $r=1$, $\delta=0$, $P(z)=1-z^{2}$, $\widetilde{P}(z)=1+z^{2}$ and $\delta_{1}=\delta_{2}=0$. The use of $(k, k, k)$ approximations give $\bar{e}(k)=k$, $\tau= 2$. If $\underline{k}(\ell)$ is the same as above and $k=k_{\ell,1}$ are $19, 26$ and $39$, then $o(k)=3k+2$ and $D(\underline{k}(\ell),z )\neq 0$ (see Appendix). By (\ref{16}), if $\lambda<2/3$, we get \[\theta(\ell)=2k+5,\quad \nu(\ell)=k-2-\lambda(3k+2).\] Now the table is the following.
\begin{center}
\begin{tabular}{|>{$}c<{$}|>{$}c<{$}|>{$}c<{$}|>{$}c<{$}|}
\hline
\ell & 1 & 2 & 3  \\
\hline
k & 19 & 26 & 39\\
\hline
\theta(\ell) & 43  & 57  & 83  \\
\hline
\nu(\ell) & 17-59\lambda  & 24-80\lambda  & 37-119\lambda   \\
\hline
\end{tabular}
\end{center}
The condition $0<\nu(1)<\nu(2) < \nu(3)<3\nu(1)$ holds, if  $\lambda<\lambda_{0}:=7/29\approx 0.241\dots$. Similarly to the above proofs we now get Theorem \ref{thm:3}.
\end{proof}

\begin{proof}[Proof of Theorem \ref{thm:rudin}]
Here we may use Theorem \ref{thm:5} with $F(z) = R(z)$ and $G(z) = R(-z)$. By (\ref{e:rudin}), we have 
\begin{align*}
2zF(z^{2})  = zF(z)+ zG(z),\quad 2zG(z^{2})  = F(z) + G(z).
\end{align*}
Therefore, we can choose $r=1, \delta = 1, P(z) = 2z$, and $\widetilde{P}(z)=1+z.$ Since $P(0) = 0$, (23) holds and we may take $\delta_{2}=0$. We use the $(k,k,k)$ approximations and we can take $\bar{e}(k)=k$ and $\tau=1$. We also choose $\underline{k}(\ell)=(k_{\ell,1}, k_{\ell,1}+1, k_{\ell,1}+2)$ where $k=k_{\ell,1}$ are $17$, $21$ and $26$. Then we get $o(k)=3k+2$ and the determinants $D(\underline{k}(\ell),z)\neq 0$ (see Appendix). Moreover, 
\[\theta(\ell)=2k+5+2\delta_1,\quad \nu(\ell)=k-2-\delta_1-\lambda(3k+2)\]
for all $\lambda < 2/3$. This gives the following table. 
\begin{center}
\begin{tabular}{|>{$}c<{$}|>{$}c<{$}|>{$}c<{$}|>{$}c<{$}|}
\hline
\ell & 1 & 2 & 3  \\
\hline
k & 17 & 21 & 26\\
\hline
\theta(\ell) & 39+2\delta_1  & 47+2\delta_1  & 57+2\delta_1  \\
\hline
\nu(\ell) & 15-\delta_1-53\lambda  & 19-\delta_1-65\lambda  & 24-\delta_1-80\lambda   \\
\hline
\end{tabular}
\end{center}
The condition $0<\nu(1)<\nu(2) < \nu(3)<2\nu(1)$ holds, if  $\lambda<3/13\approx 0.230\dots$ and $\delta_1$ is sufficiently small. If $\lambda < 21/187 \approx 0.112\dots$ and $\delta_{1}$ is small enough, then
\[
\mu = \frac{\theta(4)}{\nu(3)} = \frac{78+4\delta_1}{24-\delta_1-80\lambda}.
\]
If $21/187 \leq \lambda	< 3/13$, then 
\[
\mu = \frac{\theta(2)}{\nu(1)} = \frac{47+2\delta_1}{15-\delta_1-53\lambda}.
\]
This proves Theorem \ref{thm:rudin}, since we may choose $\delta_1$ arbitrarily small.
\end{proof}

\begin{proof}[Proof of Theorem \ref{thm:4}]
We now apply Theorem \ref{thm:5} with $F(z)=S(z)$ and $G(z)=S(z^{4})$. The use of (\ref{degree2}) gives $d=4$, $r=1$ and
\begin{align*}
F(z^{4})  = G(z),\quad G(z^{4}) = -zF(z)+(1+z+z^{2})G(z).
\end{align*}
Since $P(0)=1$, we may choose $\delta_{2}=0$ in (\ref{e3}).  We shall use of  $(k, k-1, k)$ approximations and we may take $\bar{e}(k)=k$, $\tau=1$. Again our $\underline{k}(\ell)=(k_{\ell,1}, k_{\ell,1}+1, k_{\ell,1}+2)$ and the choices for $k=k_{\ell,1}$ are $10$ and $26$. For both of these values $o(k)=3k+1$ and the determinants $D(\underline{k}(\ell),z )\neq 0$ (see Appendix). By using (\ref{16}), if $\lambda<2/3$, we get, 
\[\theta(\ell)=2k+3+\frac{2}{3}+2\delta_{1},\quad \nu(\ell)=k-1-\frac{2}{3}-\delta_{1}-\lambda(3k+1).\] 
Thus we have the following table.
\begin{center}
\begin{tabular}{|>{$}c<{$}|>{$}c<{$}|>{$}c<{$}|}
\hline
\ell & 1 & 2 \\
\hline
k & 10 & 26  \\
\hline
\theta(\ell) & 23+\frac{2}{3}+2\delta_{1}  & 55+\frac{2}{3}+2\delta_{1}  \\
\hline
\nu(\ell) & 9-\frac{2}{3}-\delta_{1}-31\lambda  & 25-\frac{2}{3}-\delta_{1}-79\lambda  \\
\hline
\end{tabular}
\end{center}
If $\lambda<\lambda_{0}:=\frac{1}{5}$ and $\delta_{1}>0$ is sufficiently small, then $0<\nu(1)< \nu(2)<4\nu(1)$. Since 
\[
\frac{55+\frac{2}{3}+2\delta_{1}}{9-\frac{2}{3}-\delta_{1}-31\lambda}>\frac{4(23+\frac{2}{3}+2\delta_{1})}{25-\frac{2}{3}-\delta_{1}-79\lambda}
\]
for all $0\leq \lambda<\lambda_{0}$, and $\delta_{1}>0$ can be arbitrarily small, Theorem \ref{thm:4} follows from Theorem \ref{thm:5}.
\end{proof}

\medskip
\noindent
\begin{tabular}{l}
Keijo V\"a\"an\"anen\\ 
Department of Mathematical Sciences,\\ 
University of Oulu, \\
P.O. Box 3000, 90014 Oulu, Finland.\medskip\\
\texttt{keijo.vaananen@oulu.fi}\\
\\
\\
\medskip\\
\end{tabular}\hfill
\begin{tabular}{l}
Wen Wu\\ 
Faculty of Mathematics and Statistics,\\
Hubei University, \\
430062 Wuhan, P.R. China.\medskip\\
Department of Mathematical Sciences, \\
University of Oulu,\\
P.O. Box 3000, 90014 Oulu, Finland.\medskip\\
\texttt{hust.wuwen@gmail.com}
\end{tabular}


\section*{Appendix}
In the appendix we shall give the values of $\Delta(\underline{k},0,z)$, where $\underline{k}=(k,k+1,k+2)$.
Tables in this appendix are organized in the following form.
\begin{longtable}{c|c|>{\small$}b{0.8\textwidth}<{$}}
\hline
\multirow{3}{*}{$k$} & \multirow{3}{*}{$o(k)$} & A_{k}(z) \\   \cline{3-3}
                     &                         & B_{k}(z) \\   \cline{3-3}
                     &                         & \text{The first two non-zero terms of }R_{k}(z)\\ \hline
\multicolumn{2}{c|}{$\underline{k}$}           & \Delta(\underline{k},0,z) \\
\hline
                     \endhead
\end{longtable}

\subsection*{Approximation polynomials and related determinants in Theorem \ref{thm:1}}
\begin{longtable}{c|c|>{\small$}b{0.8\textwidth}<{$}}
\hline
\multirow{3}{*}{29} & \multirow{3}{*}{89} & -5500\,{z}^{29}+264\,{z}^{28}-432\,{z}^{27}-432\,{z}^{26}-5908\,{z}^{
25}-144\,{z}^{24}-816\,{z}^{23}-816\,{z}^{22}-6268\,{z}^{21}-504\,{z}^
{20}-1152\,{z}^{19}-1152\,{z}^{18}-6580\,{z}^{17}-816\,{z}^{16}-1440\,
{z}^{15}-1440\,{z}^{14}+4132\,{z}^{13}-1632\,{z}^{12}-864\,{z}^{11}-
864\,{z}^{10}+4684\,{z}^{9}-1080\,{z}^{8}-336\,{z}^{7}-336\,{z}^{6}+
25048\,{z}^{5}+1992\,{z}^{4}+1224\,{z}^{3}+1224\,{z}^{2}-18708\,z-1416
 \\   \cline{3-3}
                    &                     & 8561\,{z}^{30}+9917\,{z}^{29}+18922\,{z}^{28}+20722\,{z}^{27}+24653\,{
z}^{26}+27143\,{z}^{25}+31050\,{z}^{24}+33516\,{z}^{23}+36870\,{z}^{22
}+40224\,{z}^{21}+43404\,{z}^{20}+46584\,{z}^{19}+42727\,{z}^{18}+
46075\,{z}^{17}+41528\,{z}^{16}+44186\,{z}^{15}+43421\,{z}^{14}+46979
\,{z}^{13}+45716\,{z}^{12}+48776\,{z}^{11}+40170\,{z}^{10}+43092\,{z}^
{9}+33508\,{z}^{8}+35452\,{z}^{7}+27159\,{z}^{6}+28953\,{z}^{5}+19844
\,{z}^{4}+20822\,{z}^{3}+10939\,{z}^{2}+11143\,z+528
 \\  \cline{3-3} 
                   &                    & -57640\, z^{89}+61912\, z^{90}       
                    \\ \hline
\multirow{3}{*}{30} & \multirow{3}{*}{92} & 744\,{z}^{30}-764\,{z}^{28}+844\,{z}^{26}-664\,{z}^{24}+944\,{z}^{22}-
564\,{z}^{20}+1044\,{z}^{18}-464\,{z}^{16}-444\,{z}^{14}+1064\,{z}^{12
}-544\,{z}^{10}+964\,{z}^{8}+444\,{z}^{6}-1064\,{z}^{4}-1388\,{z}^{2}+
1628
 \\   \cline{3-3}
                     &                         & 755\,{z}^{31}+755\,{z}^{30}+380\,{z}^{29}+380\,{z}^{28}+674\,{z}^{27}+
674\,{z}^{26}+129\,{z}^{25}+129\,{z}^{24}+625\,{z}^{23}+625\,{z}^{22}-
100\,{z}^{21}-100\,{z}^{20}+155\,{z}^{19}+155\,{z}^{18}-495\,{z}^{17}-
495\,{z}^{16}+318\,{z}^{15}+318\,{z}^{14}-422\,{z}^{13}-422\,{z}^{12}-
107\,{z}^{11}-107\,{z}^{10}-622\,{z}^{9}-622\,{z}^{8}-268\,{z}^{7}-268
\,{z}^{6}-553\,{z}^{5}-553\,{z}^{4}-663\,{z}^{3}-663\,{z}^{2}-598\,z-
598
 \\
\hline              
\end{longtable}

\subsection*{Approximation polynomials and related determinants in Theorem \ref{thm:2}}
\begin{longtable}{c|c|>{\small$}b{0.8\textwidth}<{$}}
\hline     
\multirow{3}{*}{29} & \multirow{3}{*}{89} & 8838\,{z}^{29}-103526\,{z}^{28}+48074\,{z}^{27}+249006\,{z}^{26}-36267
\,{z}^{25}-222096\,{z}^{24}-98200\,{z}^{23}-42694\,{z}^{22}+236901\,{z
}^{21}+813160\,{z}^{20}+82188\,{z}^{19}+105172\,{z}^{18}-76623\,{z}^{
17}-341910\,{z}^{16}-29528\,{z}^{15}+101076\,{z}^{14}+42482\,{z}^{13}+
214028\,{z}^{12}-7828\,{z}^{11}-293424\,{z}^{10}-61460\,{z}^{9}+84522
\,{z}^{8}+90968\,{z}^{7}-67100\,{z}^{6}-221774\,{z}^{5}-749016\,{z}^{4
}-82002\,{z}^{3}+59754\,{z}^{2}+157247\,z+372960
 \\   \cline{3-3}
                     &                         & 249196\,{z}^{30}+197392\,{z}^{29}+504078\,{z}^{28}-11390\,{z}^{27}+
8214\,{z}^{26}-17793\,{z}^{25}-222796\,{z}^{24}-40936\,{z}^{23}-71854
\,{z}^{22}-31047\,{z}^{21}+355216\,{z}^{20}+158208\,{z}^{19}-251360\,{
z}^{18}-245645\,{z}^{17}-403198\,{z}^{16}-44440\,{z}^{15}-75264\,{z}^{
14}+35828\,{z}^{13}-391672\,{z}^{12}-126276\,{z}^{11}+392512\,{z}^{10}
+300274\,{z}^{9}+635406\,{z}^{8}+79336\,{z}^{7}+155388\,{z}^{6}-30926
\,{z}^{5}+185780\,{z}^{4}+8222\,{z}^{3}-417854\,{z}^{2}-180793\,z-
348900
 \\   \cline{3-3}
                     &                         & -341396\, z^{89} - 904548\, z^{90}
\\ \hline
\multirow{3}{*}{30} & \multirow{3}{*}{92} & 284151\,{z}^{30}-303211\,{z}^{29}+34473\,{z}^{28}+230361\,{z}^{27}-
272325\,{z}^{26}+122016\,{z}^{25}+180112\,{z}^{24}-334273\,{z}^{23}+
109371\,{z}^{22}+36208\,{z}^{21}-196837\,{z}^{20}+184636\,{z}^{19}+
92877\,{z}^{18}-222239\,{z}^{17}+170339\,{z}^{16}+15686\,{z}^{15}-
212872\,{z}^{14}+178126\,{z}^{13}+20616\,{z}^{12}-143468\,{z}^{11}+
163829\,{z}^{10}-56575\,{z}^{9}-134101\,{z}^{8}+230702\,{z}^{7}-66928
\,{z}^{6}+9672\,{z}^{5}+97210\,{z}^{4}-157573\,{z}^{3}-40095\,{z}^{2}+
125240\,z-109296
 \\
\hline 
\end{longtable}

\subsection*{Approximation polynomials and related determinants in Theorem \ref{thm:3}}
\begin{longtable}{c|c|>{\small$}b{0.8\textwidth}<{$}}
\hline     
\multirow{3}{*}{19} & \multirow{3}{*}{59} & -153\,{z}^{19}-4\,{z}^{18}+{z}^{17}+15\,{z}^{13}+{z}^{11}+56\,{z}^{10}
+12\,{z}^{9}+15\,{z}^{7}+{z}^{5}+56\,z-4
 \\   \cline{3-3}
                     &                         & 153\,{z}^{19}-4\,{z}^{18}-{z}^{17}-15\,{z}^{13}-{z}^{11}+56\,{z}^{10}-
12\,{z}^{9}-15\,{z}^{7}-{z}^{5}-56\,z-4
 \\   \cline{3-3}
                     &                         & 2\, z^{59}+30\, z^{61}
\\ \hline
\multirow{3}{*}{20} & \multirow{3}{*}{62} & 41\,{z}^{20}+{z}^{19}-3\,{z}^{18}-4\,{z}^{14}-15\,{z}^{11}-3\,{z}^{10}
+{z}^{9}-4\,{z}^{8}-15\,{z}^{2}+z+1
 \\   \cline{3-3}
                     &                         & -41\,{z}^{20}+{z}^{19}+3\,{z}^{18}+4\,{z}^{14}-15\,{z}^{11}+3\,{z}^{10
}+{z}^{9}+4\,{z}^{8}+15\,{z}^{2}+z-1
 \\   \cline{3-3}
                     &                         & -8\,z^{62}-8\,z^{64}
\\ \hline
\multirow{3}{*}{21} & \multirow{3}{*}{65} & 76\,{z}^{21}+2\,{z}^{20}-6\,{z}^{19}-2\,{z}^{18}-8\,{z}^{15}-28\,{z}^{
12}-6\,{z}^{11}+2\,{z}^{10}-2\,{z}^{9}-28\,{z}^{3}+2\,{z}^{2}+2\,z-2
 \\   \cline{3-3}
                     &                         & -76\,{z}^{21}+2\,{z}^{20}+6\,{z}^{19}-2\,{z}^{18}+8\,{z}^{15}-28\,{z}^
{12}+6\,{z}^{11}+2\,{z}^{10}+2\,{z}^{9}+28\,{z}^{3}+2\,{z}^{2}-2\,z-2
 \\
\hline  
\end{longtable}

\subsection*{Approximation polynomials and related determinants in Theorem \ref{thm:rudin}}
\begin{longtable}{c|c|>{\small$}b{0.8\textwidth}<{$}}
\hline     
\multirow{3}{*}{17} & \multirow{3}{*}{53} & {z}^{17}+{z}^{16}-{z}^{15}+{z}^{14}-{z}^{5}-{z}^{4}-{z}^{3}+{z}^{2} \\   \cline{3-3}
                     &                         & {z}^{17}-{z}^{16}-{z}^{15}-{z}^{14}-{z}^{5}+{z}^{4}-{z}^{3}-{z}^{2} \\   \cline{3-3}
                     &                         & 16\, z^{53}-16\, z^{61}\\ \hline
\multirow{3}{*}{18} & \multirow{3}{*}{56} & 2\,{z}^{18}-6\,{z}^{17}-5\,{z}^{16}-{z}^{15}+3\,{z}^{14}+{z}^{13}-{z}^
{12}+3\,{z}^{11}-3\,{z}^{10}-{z}^{9}-{z}^{8}+3\,{z}^{7}+3\,{z}^{6}+{z}
^{5}+{z}^{4}+5\,{z}^{3}-3\,{z}^{2}-5\,z-8
 \\   \cline{3-3}
                     &                         & 2\,{z}^{18}+6\,{z}^{17}-5\,{z}^{16}+{z}^{15}+3\,{z}^{14}-{z}^{13}-{z}^
{12}-3\,{z}^{11}-3\,{z}^{10}+{z}^{9}-{z}^{8}-3\,{z}^{7}+3\,{z}^{6}-{z}
^{5}+{z}^{4}-5\,{z}^{3}-3\,{z}^{2}+5\,z-8 \\   \cline{3-3}
                     &                         & -64\,z^{56}- 64\, z^{58}\\ \hline
\multirow{3}{*}{19} & \multirow{3}{*}{59} & {z}^{19}+{z}^{18}+{z}^{17}-{z}^{16}-{z}^{15}-{z}^{14}+{z}^{13}-{z}^{12
}+{z}^{11}+{z}^{10}+{z}^{9}-{z}^{8}-{z}^{7}-{z}^{6}+{z}^{5}-{z}^{4}+2
\,{z}^{3}+2\,{z}^{2}+2\,z-2
 \\   \cline{3-3}
                     &                         & {z}^{19}-{z}^{18}+{z}^{17}+{z}^{16}-{z}^{15}+{z}^{14}+{z}^{13}+{z}^{12
}+{z}^{11}-{z}^{10}+{z}^{9}+{z}^{8}-{z}^{7}+{z}^{6}+{z}^{5}+{z}^{4}+2
\,{z}^{3}-2\,{z}^{2}+2\,z+2
 \\
\hline 
\end{longtable}

\subsection*{Approximation polynomials and related determinants in Theorem \ref{thm:4}}
\begin{longtable}{c|c|>{\small$}b{0.8\textwidth}<{$}}
\hline     
\multirow{3}{*}{10} & \multirow{3}{*}{31} & {z}^{10}-{z}^{9}-{z}^{8}+{z}^{6}-{z}^{5}-2\,{z}^{4}+{z}^{2}-z \\   \cline{3-3}
                     &                         & {z}^{9}+2\,{z}^{8}-{z}^{7}+3\,{z}^{5}+2\,{z}^{4}+1 \\   \cline{3-3}
                     &                         & z^{31}-z^{47}\\ \hline
\multirow{3}{*}{11} & \multirow{3}{*}{43} & {z}^{9}+{z}^{5}+{z}^{4}+z+1 \\   \cline{3-3}
                     &                         & -{z}^{10}-{z}^{9}-{z}^{8}-{z}^{6}-2\,{z}^{5}-{z}^{4}-{z}^{3}-{z}^{2}-2
\,z-2
 \\   \cline{3-3}
                     &                         & z^{43}-z^{64}\\ \hline
\multirow{3}{*}{12} & \multirow{3}{*}{43} & {z}^{10}+3\,{z}^{9}+{z}^{6}+{z}^{5}+{z}^{4}+{z}^{2}+z+1 \\   \cline{3-3}
                     &                         & -{z}^{11}-4\,{z}^{10}-4\,{z}^{9}-3\,{z}^{8}-{z}^{7}-2\,{z}^{6}-2\,{z}^
{5}+{z}^{4}-2\,{z}^{3}-2\,{z}^{2}-3\,z-4
 \\
\hline  
\end{longtable}

\end{document}